\documentclass[letterpaper, 10 pt, conference]{ieeeconf}  

\IEEEoverridecommandlockouts    
\usepackage{amssymb}
\usepackage{amsmath}
\usepackage{graphicx}
\usepackage{comment}
\usepackage{booktabs}
\usepackage{balance}
\usepackage{bm}
\usepackage{graphicx}
\usepackage{epstopdf}
\usepackage{amssymb}
\usepackage{wasysym}
\usepackage{float}
\usepackage{color}
\usepackage{dsfont}
\usepackage{algorithm,algorithmicx}
\usepackage{algpseudocode}
\usepackage{epstopdf}
\usepackage{tabularx}
\usepackage{cite}
\usepackage{enumerate}
\usepackage{empheq}
\newtheorem{theorem}{Theorem}

\newtheorem{property}{Property}

\newtheorem{remark}{Remark}
\newtheorem{assumption}{Assumption}
\newtheorem{lemma}{Lemma}

\newcommand{\eod}{\ensuremath{\hfill\Box}}
\IEEEoverridecommandlockouts                              
\makeatletter
\newcommand\fs@betterruled{%
	\def\@fs@cfont{\bfseries}\let\@fs@capt\floatc@ruled
	\def\@fs@pre{\vspace*{5pt}\hrule height.8pt depth0pt \kern2pt}%
	\def\@fs@post{\kern2pt\hrule\relax}%
	\def\@fs@mid{\kern2pt\hrule\kern2pt}%
	\let\@fs@iftopcapt\iftrue}
\floatstyle{betterruled}
\restylefloat{algorithm}
\makeatother

\newcommand{\mc}{\mathcal}
\newcommand{\bb}{\mathbb}

\newcommand{\col}{\operatorname{col}}

\title{\LARGE \bf
Accelerated Multi-Agent Optimization Method over Stochastic Networks
}
\author{Wicak Ananduta, Carlos Ocampo-Martinez, and Angelia Nedi\'c
\thanks{
W. Ananduta is with the Delft Center of Systems and Control (DCSC), TU Delft, the Netherlands. 
 C. Ocampo-Martinez is with Institut de Rob\`{o}tica i Inform\`{a}tica Industrial (CSIC-UPC), Barcelona, Spain. A. Nedi\'c is with  School of Electrical, Computer and Energy Engineering,
Arizona State University. 
E-mail addresses: \texttt{w.ananduta@tudelft.nl, carlos.ocampo@upc.edu, angelia.nedich@asu.edu}.  
}
}

\begin{document}

\maketitle
\thispagestyle{empty}
\pagestyle{empty}

\begin{abstract}
We propose a distributed method to solve a multi-agent optimization problem with strongly convex cost function and equality coupling constraints. The method is based on Nesterov's accelerated gradient approach and works over stochastically time-varying communication networks. We consider the standard assumptions of Nesterov's method and show that the sequence of the expected dual values converge toward the optimal value with the rate of $\mc O(1/k^2)$. Furthermore, we provide a simulation study of solving an optimal power flow problem with a well-known benchmark case.
\end{abstract}
\section{Introduction}
The advancement on information, computation and communication technologies promotes the deployment of  distributed approaches to solve complex large-scale problems, e.g., in power networks \cite{molzahn2017,yi2016} and water networks \cite{grosso2017}. On one hand, such approaches offer flexibility and scalability. On the other hand, they require more complex design than the centralized counterpart as multiple computational units must cooperate and communicate among each other.

In this paper, we deal with a multi-agent optimization problem, in which the cost function is a summation of a strongly convex cost functions. Moreover, the problem has equality coupling constraints. This formulation is mainly motivated from optimal power flow (OPF) problems of large-scale power networks \cite{molzahn2017} and resource allocation problems \cite{xiao2006,yi2016}. Furthermore, the problem can also be considered as a subclass of extended monotropic problems \cite{bertsekas2008}.  


We solve the problem in a distributed manner through its dual to deal with the coupling constraints. Particularly, we develop the method based on Nesterov's accelerated gradient method \cite{nesterov1983,beck2009}, which is an accelerated first-order approach, with the rate of $\mc O(1/k^2)$. This accelerated method has been used to develop a fast distributed gradient method to solve network utility maximization problems \cite{beck2014}, a fast alternating direction method of multipliers (ADMM) for a certain class of problems with strongly convex cost function \cite{goldstein2014}, and distributed model predictive controllers \cite{zhou2015}, among others.  

However, different from the aforementioned papers, one feature of the system that we particularly pay attention to is the time-varying nature of the communication network, over which the agents exchange information. Specifically here, we assume that the network is stochastically time-varying and this assumption can model communication failures that might occur in large-scale systems. Similar setup on communication networks can be found in \cite{wei2013,chang2016,hong2017,ananduta2019}, which develop unaccelerated first-order methods, and \cite{jakovetic2014,fercoq2015}, which propose a Nesterov-like fast gradient method for distributed optimization problem with a common decision variable. Nevertheless, whereas the former four papers do not consider an accelerated method, the latter ones deal with a different problem and work directly in the primal space. Note that different models of time-varying communication networks have also been considered, as in \cite{nedic2015,uribe2018,scutari2019}.

To summarize, the main contribution of this paper is an accelerated first-order distributed method for a multi-agent optimization problem, which works over stochastic communication networks. As a fully distributed algorithm, the parameter design and iterations only need local information, i.e., neighbor-to-neighbor communication.  Furthermore, since the method is based on Nesterov's accelerated approach, it enjoys the convergence rate of $\mc O(1/k^2)$ on the expected dual value, as shown in the convergence analysis. 

The paper is structured as follows. Section \ref{sec:setup} provides the problem setup and the cosidered model of time-varying communication networks. Afterward, Section \ref{sec:prop_met} presents the proposed distributed method along with its convergence statement. Then, in Section \ref{sec:conv}, we show the convergence analysis of the proposed method. Furthermore, we also showcase the performance of the proposed method to solve an intra-day OPF problem for a well-known benchmark case in Section \ref{sec:num_st}. Finally, Section \ref{sec:concl} concludes the paper by providing some remarks and discussions about future work. 

\subsection*{Notation and properties}
The set of real numbers is denoted by $\bb R$. For any $a\in \bb R$, $\bb R_{\geq a}$ denotes $\{b\in \bb R: b \geq a \}$. The inner product of vectors $x,y \in \bb R^n$ is denoted by $\langle x,y\rangle$, whereas the Euclidean vector norm and the induced matrix norm are denoted by $\|\cdot\|$. The operator $\col \{\cdot \}$ stacks the arguments column-wise. We use $0_n$ to denote zero vector with dimension $n$. When the dimension is clear from the context, we may omit the subscript. 
Furthermore, the following properties will be used in the convergence analysis.

\begin{property}[Strong convexity]
	\label{def:strong_cvx}
	A differentiable function $f(x):\bb{R}^n\to\bb{R}$ is strongly convex, if for any $x,y\in\bb{R}^n$ it holds that
	$$ \langle\nabla f(y)-\nabla f(x),y-x\rangle \geq \sigma \|y-x\|^2, $$
	where $\sigma$ is the strong convexity constant. \eod
\end{property}
\begin{property}[Lipschitz smoothness]
	\label{def:smooth}
	A function $f(x):\bb{R}^n\to\bb{R}$ is continuously differentiable with Lipschitz continuous gradient, if for any $x,y\in \bb R^n$ it holds that
	$$\|\nabla f(y)-\nabla f(x)\| \leq L \|y-x\|, $$
	where  $L$ denotes the Lipschitz constant. \eod
\end{property}
\section{Problem setup}
\label{sec:setup}

\subsection{Multi-agent optimization problem}
We consider a multi-agent system, where the set of agents is denoted by $\mc N := \{1,2,\dots,N \}$. The agents want to cooperatively solve an optimization problem in the following form:
\begin{subequations}
	\begin{align}
	\underset{u_i \in \mc U_i,  \forall i\in \mc N}{\operatorname{minimize}} \ \ & \sum_{i=1}^{N} f_i(u_i)  \label{eq:cost_f}\\
	\text{s.t.} \ \ &  G_i^i u_i + \sum_{j\in \mc N_i} G_i^j u_j = g_i, \quad \forall i \in \mc N, \label{eq:coup_con}
	\end{align}
	\label{eq:prob}%
\end{subequations}
where $u_i \in \bb R^{n_i}$ and $\mc U_i \in \bb R^{n_i}$ denote the decision vector and the local set constraint of agent $i$, respectively. In \eqref{eq:cost_f}, each cost function $f_i(u_i)$ is associated to agent $i$. Moreover, each equality in \eqref{eq:coup_con}, with the non-zero matrix $G_i^j \in \bb R^{m_i\times n_j}$, for each $j\in \mc N_i \cup \{i\}$ and $i\in \mc N$, and $g_i \in \bb R^{m_i}$, is assigned to agent $i$ and couples agent $i$ with some other agents, i.e., $j \in \mc N_i \subseteq \mc N$. 
Based on the formulation of the coupling constraints in \eqref{eq:coup_con}, we can represent the system as a directed graph, denoted by $\mc S = (\mc N, \mc V)$, where $\mc V$ denotes the set of links that represents how each agent influences the coupling constraint \eqref{eq:coup_con} of other agents. Specifically, the link  $(j,i) \in \mc V$ implies that $u_j$ appears on the coupling constraint of agent $i$, i.e., $j\in \mc N_i$. Therefore, we can say that $\mc N_i$ is the set of in-neighbors of agent $i$. On the other hand, we also introduce the set of out-neighbors, denoted by $\mc M_i$, i.e., $\mc M_i = \{j\in \mc N: (i,j) \in \mc V \}$. Furthermore, we define $i \in \mc M_i$ and, in general, $\mc M_i$ may not be equal to $\mc N_i \cup \{i\}$  (see Figure \ref{fig:ex}).


Problem \eqref{eq:prob} is a subclass of the extended monotropic problem \cite{bertsekas2008}. Resource allocation problems \cite{xiao2006,yi2016} can also be formulated as in \eqref{eq:prob}. A particular practical problem of interest, which can be represented by \eqref{eq:prob}, is the direct current (DC) OPF problem \cite{molzahn2017},  where the decision vectors $u_i$ might consist of the real powers and phase angle, whereas \eqref{eq:coup_con} represents the DC approximation of the power flow equations. Note that, in the DC-OPF problem, $\mc M_i = \mc N_i \cup \{i\}$.
 
Now, we consider the following assumptions hold.
\begin{assumption}
	\label{as:cost_f}
	The function $f_i : \bb R^{n_i} \to \bb R$, for each $i\in \mc N$, is differentiable and strongly convex with strong convexity parameter denoted by $\sigma_i$. \eod
\end{assumption}
\begin{assumption}
	\label{as:loc_set}
	The local set $\mc U_i$, for each $i\in \mc N$, is compact and convex. \eod
\end{assumption}
\begin{assumption}
	\label{as:feas_set}
	The feasible set of Problem \eqref{eq:prob} is non-empty. \eod
\end{assumption}

\begin{figure}
	\centering
	\vspace{5pt}
	\includegraphics[scale=1]{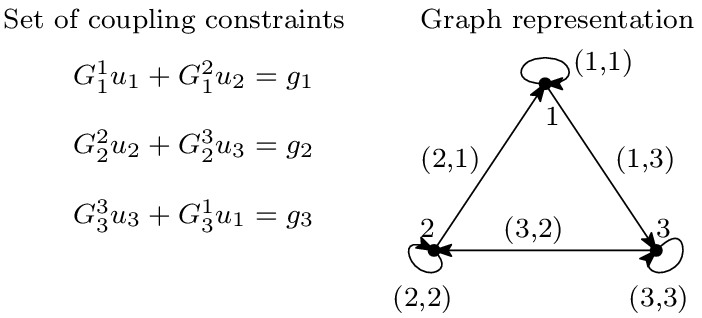}
	\caption{A small network of three agents. Notice that $\mc N_1 = \{2\}$ and $\mc M_1 = \{1,3\}$.
	}
	\label{fig:ex}
\end{figure}

Assumptions \ref{as:cost_f} and \ref{as:loc_set} are rather restrictive, however, commonly used in the applications considered, i.e., OPF and resource allocation problems. Moreover, these assumptions allow us to apply Nesterov's accelerated gradient method to solve the dual problem of \eqref{eq:prob}, as these assumptions result in a dual function with Lipschitz continuous gradient. This statement is elaborated further in Section \ref{sec:conv}. Furthermore, Assumption \ref{as:feas_set} is considered to ensure that the proposed algorithm can find a solution to Problem \eqref{eq:prob}.

\subsection{Stochastic communication networks}
\label{sec:tv_net}
The aim of this work is to design a distributed optimization algorithm that solves Problem \eqref{eq:prob}. As a distributed method, the algorithm requires each agent to communicate with other agents over a communication network, which we suppose to be time-varying. 
Precisely, the communication network is represented by the undirected graph $\mc G(k) = (\mc N, \mc L(k))$, where $\mc L(k)\subseteq \mc N \times \mc N$ denotes the set of communication links that may vary over iteration $k$, i.e., $\{i,j\} \in \mc L(k)$ implies that agents $i$ and $j$ can communicate at iteration $k$. Thus, we denote by $\mc E_i(k)$ the set of agents that can exchange information with agent $i$, i.e., $\mc E_i(k) = \{j\in\mc N: \{i,j\} \in \mc L(k) \}$. Furthermore, we consider the activation of communication links as a random process and the following assumption holds.
\begin{assumption}
	\label{as:com_link}
	The set $\mc L(k)$ is a random variable that is independent and identically distributed across iterations. Furthermore, any communication link of neighboring agents is active with a positive probability denoted by $\beta_{\{i,j\}}$, i.e., $\bb P(\{i,j\} \in \mc L(k)) = \beta_{\{i,j\}} >0$, for $\{i,j\} \in \{\{i',j'\} \in \mc N \times \mc N: j' \in \mc N_{i'}, i' \in \mc N \}$. Additionally, $\beta_{\{i,i\}}=1$, for all $i\in\mc N$. \eod
\end{assumption}

Assumption \ref{as:com_link} implies that the probability that agent $i$ can receive information from all its in-neighbors $j\in \mc N_i$ at the same iteration $k$ is positive. Let $\alpha_i$ denote this probability, thus we have that $\alpha_i = \prod_{j\in \mc N_i} \beta_{\{i,j\}}$.  
\section{Proposed method}
\label{sec:prop_met}
In this section, we propose a distributed method to solve Problem \eqref{eq:prob} over stochastic communication networks. The proposed method actually solves the dual problem associated to \eqref{eq:prob} and is based on Nesterov's accerelated gradient approach \cite{nesterov1983,beck2009}.

To that end, let $\lambda_i \in \bb R^{m_i}$ denote the Lagrange multiplier associated to \eqref{eq:coup_con}, for each $i\in\mc N$, and $\lambda = \col\{\lambda_i, i \in \mc N \}$. Thus, we define the dual function, associated to \eqref{eq:prob} and denoted by $q(\lambda)$, as follows:
\begin{equation}
	q(\lambda) = \sum_{i \in \mc N} q_i(\lambda^i) , \label{eq:d_func}
\end{equation}
where
\begin{equation}
	q_i(\lambda^i) = \min_{u_i \in \mc U_i} \left\{ f_i(u_i) - \langle\lambda_i, g_i\rangle + \sum_{j\in \mc M_i} \langle G_j^{i\top}\lambda_j, u_i \rangle \right\}. \label{eq:loc_d_func}
\end{equation}
Note that $\lambda^i$ denotes all Lagrange multipliers associated to the coupling constraints that involve agent $i$, i.e., $\lambda^i = \col\{\lambda_j, j\in \mc M_i \}$. 
We will then solve the dual problem:
\begin{equation}
	\operatorname{maximize} q(\lambda), \label{eq:d_prob}
\end{equation}
by adapting Nesterov's accelerated gradient method such that it works over stochastically time-varying communication networks (c.f. Section \ref{sec:tv_net}). Note that, due to Assumptions \ref{as:cost_f}-\ref{as:feas_set}, the strong duality holds \cite[Proposition 5.2.1]{bertsekas1995}.

\begin{algorithm}[!t]
	\caption{Distributed accelerated method}
	\label{alg:std}
	\textbf{Initialization} (for each $i\in\mc N$)\\
	Set $\theta(1)=1$
	and $\hat{\lambda}_i(1) = \lambda_i(0)  =0$\\
	\textbf{Iteration} (for each $i\in\mc N$, $k\geq 1$)
	\begin{enumerate}
		\item Compute $u_i(k)$:
		\begin{equation}
			u_i(k) = \arg\min_{u_i \in \mc U_i} f_i(u_i) + \sum_{j\in \mc M_i} \langle G_j^{i\top}\hat{\lambda}_j(k), u_i \rangle
			\label{eq:u_std}
		\end{equation}
		\item Send $G_j^i u_i(k)$ to out-neighbors $j \in \mc M_i$ and receive $G_i^j u_j(k)$ from the in-neighbors $j\in \mc N_i$
		\item Compute $\lambda_i(k) $:
		\begin{equation}
		\label{eq:lambda_std}
		\lambda_i(k) = 
		\hat{\lambda}_i(k) + \eta_i \left(G_i^i u_i(k)+\sum_{j\in \mc N_i} G_i^j u_j(k)-g_i\right)
		\end{equation}
		\item Compute $\theta(k+1) = \frac{1+\sqrt{1+4\theta^2(k)}}{2}$
		\item Compute $\hat{\lambda}_i(k+1)$:
		\begin{equation}
		\label{eq:lambda_h_std}
		\hat{\lambda}_i(k+1) = \lambda_i(k) + \frac{\theta(k)-1}{\theta(k+1)}\left(\lambda_i(k)-\lambda_i(k-1) \right)
		\end{equation}
		\item Send $\hat{\lambda}_i(k+1)$ to in-neighbors $j \in  \mc N_i$ and receive $\hat{\lambda}_j(k+1)$ from the out-neighbors $j\in  \mc M_i$
	\end{enumerate}

\end{algorithm}

Hence, first we state  the distributed method based on Nesterov's accelerated gradient approach without considering stochastic communication networks, i.e., the information required to perform the updates is always available. The method is shown in Algorithm \ref{alg:std}. For a detailed design procedure of Nesterov's accelerated method, the reader might check \cite{beck2009,beck2014}. The main steps in the iteration of Nesterov's accelerated approach can be seen in Steps 4 and 5 where an interpolated point of each Lagrange multiplier $\lambda_i$ (denoted by $\hat{\lambda}_i$) is computed. As a distributed method, these steps are carried out by each agent. Furthermore, the step-size of the gradient ascent in \eqref{eq:lambda_std}, denoted by $\eta_i$, is a local variable that must be chosen appropriately (c.f. Theorem \ref{th:conv}). Finally, note that, in \eqref{eq:u_std}, $u_i$ is updated by solving a local minimization derived from \eqref{eq:loc_d_func} and based on the interpolated points of the Lagrange multipliers from the out-neighbors, i.e., $\hat{\lambda}^i = \col\{\hat{\lambda}_j, j\in\mc M_i \}$. Due to Assumptions \ref{as:cost_f} and \ref{as:loc_set}, the local minimization in Step 1 admits a unique solution.

\begin{algorithm}[!t]
	\caption{Distributed accelerated method over stochastic networks}
	\label{alg:prop}
	\textbf{Initialization} (for each $i\in\mc N$)\\
		Set $\theta(1)=1$,
		$\lambda_i(0)=0$, and
		$\hat{\xi}_j^i(1) ={\xi}_j^i(0) =0,$ for all $j\in \mc M_i$\\
	\textbf{Iteration} (for each $i\in\mc N$, $k\geq 1$): with random realization of $\mc L(k)$
	\begin{enumerate}
		\item Compute $u_i(k)$:\vspace{-3pt}
		\begin{equation}
		u_i(k) = \arg\min_{u_i \in \mc U_i} f_i(u_i) + \sum_{j\in \mc M_i} \langle G_j^{i\top}\hat{\xi}_j^i(k), u_i \rangle
		\label{eq:u_tv}
		\end{equation}
		\item Send $G_j^i u_i(k)$ to out-neighbors $j \in \mc E_i(k) \cap \mc M_i$ and receive $G_i^j u_j(k)$ from in-neighbors $j\in \mc E_i(k) \cap \mc N_i$
		\item Compute $\lambda_i(k)$:\vspace{-3pt}
		\begin{equation}
			\label{eq:lambda_tv}
		\hspace{-15pt}	\lambda_i(k) = 
			\begin{cases}
			\hat{\xi}^i_i(k) + \eta_i \left(G_i^i u_i(k)+\sum_{j\in \mc N_i} G_i^j u_j(k){-g_i}\right), \\
			\qquad \quad \  \text{if } \mc N_i \subseteq \mc E_i(k)\\
			\hat{\xi}_i^i(k), \quad \text{otherwise}
			\end{cases}
		\end{equation}
		\item Send ${\lambda}_i(k)$ to in-neighbors $j \in \mc E_i(k) \cap \mc N_i$ and receive ${\lambda}_j(k)$ from out-neighbors $j\in \mc E_i(k) \cap \mc M_i$
		\item Update $\xi_j^i(k) $, for all $j\in\mc M_i$:\vspace{-3pt}
		\begin{equation}
		\xi_j^i(k) = 
		\begin{cases}
		{\lambda}_j(k), \ \text{for } j \in \mc M_i \cap \mc E_i(k),\\
		\hat{\xi}_j^i(k), \quad \text{otherwise}
		\end{cases}
		\label{eq:xi}
		\end{equation}
		\item Compute $\theta(k+1) = \frac{1+\sqrt{1+4\theta^2(k)}}{2}$
		\item Compute $\hat{\xi}_j^i(k+1)$, for all $j\in\mc M_i$:\vspace{-3pt}
		\begin{equation}
		\hspace{-10pt}
			\label{eq:lambda_h_tv}
			\hat{\xi}_j^i(k+1) = \xi_j^i(k) + \frac{\theta(k)-1}{\theta(k+1)}\left(\xi_j^i(k)-\xi_j^i(k-1) \right)
		\end{equation}
	\end{enumerate}
	
\end{algorithm}
Now, we are ready to state the proposed method, which works over stochastic communication networks. The method is shown in Algorithm \ref{alg:prop}.  We adjust the gradient step update (Step 3) in order to take into account the time-varying nature of the communication network. As can be seen in Step 3,  $\lambda_i$ is only updated with the gradient step when agent $i$ receives new information from all in-neighbors in $\mc N_i$. Furthermore, the required Lagrange multipliers from the other agents $j\in \mc M_i$ are tracked by agent $i$ using the auxiliary vector $\xi^i=\col\{\xi_j^i, j\in \mc M_i \}$, where each $\xi_j^i$ is updated in \eqref{eq:xi}. Additionally, each agent $i$ must compute the interpolated point of $\xi_j^i$, denoted by $\hat{\xi}_j^i$ in \eqref{eq:lambda_h_tv}. This step is different than the steps in Algorithm \ref{alg:std}, where the exchanged information is actually the interpolated point $\hat{\lambda}_i$. 

The outcome of Algorithm \ref{alg:prop}, which is the main result of this work, is stated as the following theorem.
\begin{theorem}
\label{th:conv}
Let Assumptions \ref{as:cost_f}-\ref{as:com_link} hold and the sequence $\lambda(k)$ be generated by Algorithm \ref{alg:prop} with $\eta_i \in (0, 1/L_i]$, where $L_i$ is defined as follows:
\begin{equation}
\label{eq:Li}
L_i = \sum_{j\in \mc N_i \cup \{i\} }\frac{\|G^j\|^2}{\sigma_j},
\end{equation}
in which $G^j = \col\{G_i^j, i\in \mc M_j \}$ and $\sigma_j$ is the strong convexity constant of $f_j(u_j)$.
Furthermore, let $q(\lambda)$ be defined by \eqref{eq:d_func} and $\lambda^{\star}$ be an optimal solution of the dual problem \eqref{eq:d_prob}. Then, 
\begin{enumerate}
	\item It holds that
	\begin{equation}
	\bb E \left(q(\lambda^{\star}) - q(\lambda(k))\right) \leq \frac{C}{(k+1)^2},
	\label{eq:conv}
	\end{equation}
	where $C$ is a non-negative constant.
	\item Hence, it also holds that
	\begin{equation}
		\lim_{k\to \infty} \bb E \left(q(\lambda^{\star}) - q(\lambda(k))\right) = 0,
		\label{eq:conv2}
	\end{equation}
	almost surely. \eod
\end{enumerate}
\end{theorem}

Theorem \ref{th:conv} shows that the expected dual values converge to the optimal dual value with the rate of $\mc O(1/k^2)$. Furthermore, the choice of parameter $\eta_i$, for each agent $i\in \mc N$, which is sufficient to achieve convergence, can be obtained locally, i.e., agent $i$ only requires some information from its in-neighbors in $\mc N_i$ (see \eqref{eq:Li}). 
\section{Convergence analysis}
\label{sec:conv}
First, Section \ref{sec:pre_res} provides some preliminary results, which become the building blocks to prove Theorem \ref{th:conv}. Then, the proof of Theorem \ref{th:conv} is given in Section \ref{sec:proof_th}.
\subsection{Preliminary results}
\label{sec:pre_res}
First, we show that the local dual function, $q_i(\lambda^i)$, for any $i\in\mc N$, is a Lipschitz smooth function.
\begin{lemma}
	\label{le:Lips_smooth}
	Let Assumptions \ref{as:cost_f}-\ref{as:feas_set} hold. The local dual function $q_i(\lambda^i)$ defined in \eqref{eq:loc_d_func} is Lipschitz smooth with Lipschitz constant $\frac{\|G^i\|^2}{\sigma_i}$. \eod
\end{lemma}
\begin{proof}
	Recall the definition of $q_i(\lambda^i)$ in \eqref{eq:loc_d_func} and let $u_i(\lambda^i) = \arg\min_{u_i \in \mc U_i}\hspace{-3pt}\left\{ f_i(u_i) + \sum_{j\in \mc M_i} \langle G_j^{i\top}\lambda_j, u_i \rangle \right\}$ and $v_i(\mu^i) \hspace{-3pt}= \arg\min_{u_i \in \mc U_i}\left\{ f_i(u_i) + \sum_{j\in \mc M_i} \langle G_j^{i\top}\mu_j, u_i \rangle \right\}$. Since $u_i(\lambda^i),v_i(\mu^i) \in \mc U_i$, the optimality conditions \cite{nedic2008} of the preceding minimizations  yield the following inequalities:
	\begin{align}
	0 &\leq  \langle \nabla f_i(u_i(\lambda^i)) +   G^{i\top}\lambda^i, v_i(\mu^i) - u_i(\lambda^i) \rangle,  \label{eq:ineq1}\\
	0 &\leq  \langle \nabla f_i(v_i(\mu^i)) +   G^{i\top}\mu^i, u_i(\lambda^i) - v_i(\mu^i) \rangle. \label{eq:ineq2}
	\end{align}
	Combining \eqref{eq:ineq1} and \eqref{eq:ineq2} gives
	\begin{align}
		0 & \leq \langle \nabla f_i(u_i(\lambda^i)) -\nabla f_i(v_i(\mu^i)), v_i(\mu^i) - u_i(\lambda^i) \rangle  
		\notag\\
		& \quad + \langle G^{i\top}(\lambda^i - \mu^i), v_i(\mu^i) - u_i(\lambda^i) \rangle \notag\\
		& \leq -\sigma_i \|v_i(\mu^i) - u_i(\lambda^i) \|^2 \notag\\
		& \quad + \langle \lambda^i - \mu^i,  G^{i}(v_i(\mu^i) - u_i(\lambda^i)) \rangle, \label{eq:ineq3}
	\end{align}
	where the second inequality is obtained since $f_i(\cdot)$ is strongly convex (c.f. Property \ref{def:strong_cvx}). Furthermore, the strong convexity of $f_i(\cdot)$ also implies that $u_i(\lambda^i)$ is unique and $q_i(\lambda^i)$ is differentiable, with $\nabla q_i(\lambda^i) = G^i u_i(\lambda^i)-\tilde{g}^i$, where $\tilde{g}^i = \col\{\tilde{g}_j^i,j\in\mc M_i \}$ and $\tilde{g}_j^i = 0_{m_j}$ if $j\neq i$ and $\tilde{g}_j^i=g_i$ otherwise. Thus,$\nabla q_i(\mu^i) - \nabla q_i(\lambda^i)= G^i(v_i(\mu^i) - u_i(\lambda^i))$. Using \cite[Lemma 1.1]{beck2014} we obtain that
	\begin{align}
		\frac{1}{\|G^i\|^2} \|\nabla q_i(\mu^i) - \nabla q_i(\lambda^i)\|^2 \leq \|v_i(\mu^i) - u_i(\lambda^i)\|^2. \label{eq:ineq4}
	\end{align}
	By adding $\langle \lambda^i - \mu^i, \tilde{g}^i-\tilde{g}^i \rangle=0$ to the right-hand side of \eqref{eq:ineq3}, and then rearranging \eqref{eq:ineq3} as well as using \eqref{eq:ineq4} and the fact that $G^i v_i(\mu^i)-\tilde{g}^i= \nabla q_i(\mu^i)$ and $G^i u_i(\lambda^i)-\tilde{g}^i = \nabla q_i(\lambda^i)$, we obtain that
	\begin{align*}
		&\frac{\sigma_i}{\|G^i\|^2} \|\nabla q_i(\mu^i) - \nabla q_i(\lambda^i)\|^2 \\
		&\leq \langle {\lambda^i - \mu^i}, \nabla q_i(\mu^i) - \nabla q_i(\lambda^i)\rangle\\
		&\leq \|\mu^i-\lambda^i \|\|\nabla q_i(\mu^i) - \nabla q_i(\lambda^i)\|,
	\end{align*}
	where the second inequality is obtained using the Cauchy-Schwarz inequality. Thus, we have that
	\begin{equation*}
		\|\nabla q_i(\mu^i) - \nabla q_i(\lambda^i)\| \leq \frac{\|G^i\|^2}{\sigma_i}\|\mu^i-\lambda^i\|, 
	\end{equation*}
	showing that $q_i(\cdot)$ is Lipschitz smooth with Lipschitz constant $\frac{\|G^i\|^2}{\sigma_i}$ (c.f. Property \ref{def:smooth}).
\end{proof}
\begin{remark}
	The Lipschitz constant of $q_i(\cdot)$ can be computed locally by each agent $i \in \mc N$ since  $G^i$ and  parameter $\sigma_i$ are local information. \eod
\end{remark} 

\begin{lemma}
	\label{le:desc}
	Let Assumptions \ref{as:cost_f}-\ref{as:feas_set} hold. For any $\mu,\lambda \in \bb R^{\sum_{i \in \mc N}m_i}$, it holds that
	\begin{equation}
		q(\lambda) \geq q(\mu) + \langle \lambda-\mu, \nabla q(\mu) \rangle - \sum_{i \in \mc N} \frac{L_i}{2} \| \lambda_i - \mu_i\|^2,
	\end{equation}
	where $L_i$, for each $i\in\mc N$, is defined in \eqref{eq:Li}.

\end{lemma}
\begin{proof}
	Since $q_i(\lambda_i)$ is concave and has a Lipschitz smooth gradient (Lemma \ref{le:Lips_smooth}), it follows from \cite{zhou2018} that 
	\begin{equation}
		q_i(\lambda^i) \geq q_i(\mu^i) + \langle \lambda^i-\mu^i, \nabla q_i(\mu^i)  \rangle - \frac{\|G^i\|^2}{2\sigma_i} \| \lambda^i - \mu^i\|^2. \label{eq:desc_loc}
	\end{equation}
	The desired inequality follows by summing \eqref{eq:desc_loc} over $i\in \mc N$.
\end{proof}

The Lipschitz smoothness property of the dual function (Lemma \ref{le:desc}) is sufficient to show the inequality \eqref{eq:main_ineq} stated in Lemma \ref{le:main_ineq}, which will become the key to prove Theorem \ref{th:conv}. Note that Lemma \ref{le:main_ineq} is similar to \cite[Lemma 4.1]{beck2009} and \cite[Lemma 5]{goldstein2014}, although, differently from these references,  the step-size $\eta_i$ in \eqref{eq:lambda_std} does not need to be the Lipschitz constant of the (dual) function.
\begin{lemma}
	\label{le:main_ineq}
	Let Assumptions \ref{as:cost_f}-\ref{as:feas_set} hold and the sequence $\{\theta(k), u_i(k), \lambda_i(k),\hat{\lambda}_i(k), \forall i \in \mc N  \}$ be generated by Algorithm \ref{alg:std}, with $\eta_i \in (0, 1/L_i]$, where $L_i$ is defined by \eqref{eq:Li}. Furthermore, let $\lambda^{\star}= \col\{\lambda_i^{\star}, i\in \mc N \}$ be an optimal solution of the dual problem \eqref{eq:d_prob} and define $\omega_i(k)$ by
	\begin{equation}
		\label{eq:omega}
		\omega_i(k) = \theta(k)\lambda_i(k)- (\theta(k)-1)\lambda_i(k-1)-\lambda_i^{\star},
	\end{equation} 
	for each $i\in \mc N$. Then, it holds that
	\begin{equation}
		\label{eq:main_ineq}
		\begin{aligned}
		& \sum_{i \in \mc N} \frac{1}{2\eta_i}\left(\|\omega_i(k+1)\|^2 - \|\omega_i(k) \|^2\right) \leq \\
		& \quad (\theta(k))^2(q(\lambda^{\star})-q(\lambda(k)))\\
		& \quad -(\theta(k+1))^2(q(\lambda^{\star})-q(\lambda(k+1))).
		\end{aligned}	
	\end{equation}
\end{lemma}
\smallskip
\begin{proof}
	see Appendix \ref{sec:app1}.
\end{proof}
\subsection{Proof of Theorem \ref{th:conv}}
\label{sec:proof_th}
Recall that $\alpha_i$ is the probability that the communication links between agent $i$ and all its in-neighbors $j\in \mc N_i$ are active, i.e., $\alpha_i = \prod_{j\in \mc N_i} \beta_{\{i,j\}}$ and introduce the following function $V(k)$:
\begin{equation}
	V(k) = \sum_{i \in \mc N} \frac{1}{2\alpha_i \eta_i}\|\omega_i(k)\|^2,
	\label{eq:V}
\end{equation}
where $\omega_i(k)$ is defined in \eqref{eq:omega}. 

To show the convergence, first we evaluate the sequence $\{\bb E(V(k))\}$. To this end, define $\mc F(k)$ as the filtration up to and including iteration $k$, i.e., $\mc F(k) = \{\mc L(\ell), \lambda(\ell), \xi(\ell),\  \ell = 0,1,2,\dots,k  \}$, where $\xi(k) = \col\{\xi^i(k), i\in \mc N \}$. Based on \eqref{eq:lambda_tv}, $\lambda_i(k)$, for each $i\in\mc N$, is updated with the gradient ascent rule only when all the in-neighbors of agent $i$ in $\mc N_i$ send new information to agent $i$. Otherwise, $\lambda_i(k)= \hat{\xi}_i^i(k)$. 
 Therefore, if $\mc N_i \subseteq \mc E_i(k+1)$, $\omega_i(k+1)$ is computed using $\lambda_i(k+1)$ updated with the gradient ascent step. Otherwise, since $\lambda_i(k+1) = \hat{\xi}_i^i(k+1)$ (c.f. \eqref{eq:lambda_tv}), we have that
\begin{align*}
	\omega_i(k+1) &= \theta(k+1)\hat{\xi}_i^i(k+1)- (\theta(k+1)-1)\lambda_i(k)-\lambda_i^{\star}\\
	&=\theta(k+1)\lambda_i(k)+ (\theta(k)-1)(\lambda_i(k)-\lambda_i(k-1)) \\
	&\quad 	- (\theta(k+1)-1)\lambda_i(k)-\lambda_i^{\star}\\
	&
	= \omega_i(k),
\end{align*}
where the second equality is obtained by using \eqref{eq:lambda_h_tv} and  since $\lambda_i(k)=\xi_i^i(k)$, for any $k \geq 0$, due to \eqref{eq:xi} and a proper initialization in Algorithm \ref{alg:prop}. 

Thus, we can see that $\omega(k+1)$ is updated with probability $\alpha_i$ and remains the same, i.e., $\omega_i(k+1)=\omega_i(k)$ with probability $1-\alpha_i$. Based on this fact, we obtain, with probability 1, that
\begin{align}
	&\bb E\left(V(k+1)-V(k)| \mc F(k) \right) \notag \\
	&=\bb E \left(\sum_{i \in \mc N}\frac{1}{2\alpha_i \eta_i}\left(\|\omega_i(k+1)\|^2-\|\omega_i(k)\|^2\right) \Bigg| \mc F(k) \right) \notag\\
	&=\sum_{i \in \mc N}\frac{1}{2\eta_i}\left(\frac{\alpha_i}{\alpha_i}\|\omega_i(k+1)\|^2 + \frac{1-\alpha_i}{\alpha_i}\|\omega_i(k)\|^2 \right. \notag\\
	& \qquad \quad  \left. - \frac{1}{\alpha_i}\|\omega_i(k)\|^2  \right) \notag \\
	&= \sum_{i \in \mc N}\frac{1}{2\eta_i} \left(\|\omega_i(k+1)\|^2 - \|\omega_i(k) \|^2\right) \notag\\	
	&\leq (\theta(k))^2(q(\lambda^{\star})-q(\lambda(k))) \notag\\
	& \quad -(\theta(k+1))^2(q(\lambda^{\star})-q(\lambda(k+1))),
	\label{eq:ineq5}
\end{align}
where the inequality is obtained based on \eqref{eq:main_ineq} in Lemma \ref{le:main_ineq}. Iterating \eqref{eq:ineq5}, for $\ell =1,2,\dots,k-1$, and taking the total expectation, we have that
\begin{align}
	&\bb E\left(\sum_{\ell=1}^{k-1} \left(V(\ell+1)-V(\ell)\right) \right) \notag\\
	&\leq \bb E \Bigg(\sum_{\ell=1}^{k-1}(\theta(\ell))^2(q(\lambda^{\star})-q(\lambda(\ell))) \Bigg. \notag\\
	&\Bigg. \quad -(\theta(\ell+1))^2(q(\lambda^{\star})-q(\lambda(\ell+1))) \Bigg) \notag\\
	& \iff \bb E(V(k) - V(1)) \leq \theta(1)^2 \bb E \left( q(\lambda^{\star})-q(\lambda(1)) \right) \notag \\
	& \qquad -\bb E ( \theta(k))^2(q(\lambda^{\star})-q(\lambda(k))) ). \label{eq:ineq6}
\end{align}
Rearranging the inequality in \eqref{eq:ineq6} yields
\begin{align}
&	\bb E \left( \theta(k)^2(q(\lambda^{\star})-q(\lambda(k))) \right) \notag\\ 
&	\leq \bb E( V(1) - V(k)) +\theta(1)^2 \bb E \left(q(\lambda^{\star})- q(\lambda(1)) \right) \notag\\
&	\leq \bb E (V(1) + q(\lambda^{\star})- q(\lambda(1))), \label{eq:ineq7}
\end{align}
where the second inequality is obtained since $\theta(1)=1$ and by dropping $-\bb E(V(k))$ since it is non-positive for any $k\geq1$. Finally, note that $\theta(k)$ is not random and it holds that $\theta(k) \geq \frac{k+1}{2}$ since $\theta(1)=1$ and it is updated using the equation in step 6 of Algorithm \ref{alg:prop} \cite{beck2009}. Using this fact and \eqref{eq:ineq7}, the desired inequality \eqref{eq:conv} follows, where $ C = 4\bb E \left(V(1) + q(\lambda^{\star})- q(\lambda(1))\right) \geq 0,$ since $\bb E(V(k))\geq0$, for any $k\geq 1$, and $q(\lambda^{\star}) =\max_{\lambda}q(\lambda)$, thus $\bb E(q(\lambda^{\star})- q(\lambda(1)) \geq 0$. 
Upon obtaining \eqref{eq:conv}, we can show the equality \eqref{eq:conv2}. 	Since $C$ in \eqref{eq:conv} is non-negative, the term	$\bb E \left(q(\lambda^{\star}) - q(\lambda(k))\right)$ converges to 0. Furthermore, using the Markov inequality, for any $\delta \in \bb R_{>0}$, we have that
	$\limsup_{k \to \infty} \bb P(q(\lambda^{\star}) - q(\lambda(k) \geq \delta) 
	\leq \limsup_{k \to \infty} \frac{1}{\delta}\bb E(q(\lambda^{\star}) - q(\lambda(k) ) = 0,$ 
thus, $\lim_{k\to \infty} \bb E \left(q(\lambda^{\star}) - q(\lambda(k))\right) = 0$, almost surely.
\eod
\section{Numerical study}

\label{sec:num_st}
\begin{figure}
	\centering
	\includegraphics[scale=0.9]{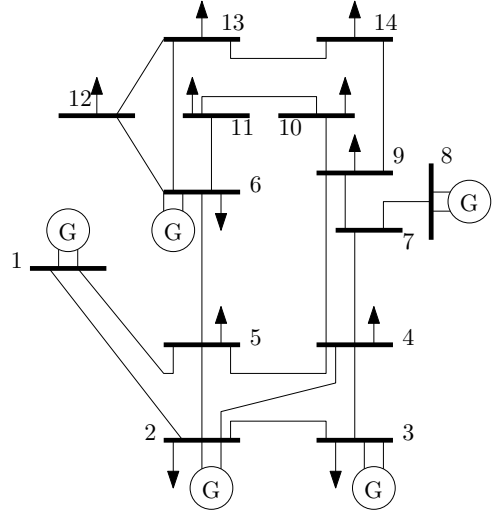}
	\caption{The IEEE 14-bus network.
	}
	\label{fig:14-bus}
\end{figure}

We use the IEEE 14-bus benchmark case, which is shown in Figure \ref{fig:14-bus}, as the test case in this simulation study, where we solve an intra-day DC-OPF problem, with time horizon ($h$) of 6 hourly steps. We suppose that each bus is an agent in the network, though there are only five active agents, which have the capability of generating power, bounded by the capacity of the generators. Furthermore, we consider the DC-approximation of the power flow equations, as follows:
\begin{equation}
P^{\rm g}_{i,t} - P^{\rm l}_{i,t} = \sum_{j\in \mc N_i}B_{\{i,j\}}(\psi_{i,t} - \psi_{j,t}), \ \forall i \in \mc N, t=1,\dots,h,
\label{eq:dc_pf}
\end{equation}
where $P^{\rm g}_{i,t} \in \bb R_{\geq 0}$  denotes the power generated at bus $i$ at time step $t$, $P^{\rm l}_{i,t} \in \bb R_{\geq 0}$ denotes the power demand assumed to be known for the whole time horizon, $B_{\{i,j\}}$ denotes the susceptance of line $\{i,j\}$,  whereas $\psi_i$ denotes the phase angle of bus $i$.
The equalities in \eqref{eq:dc_pf} become the coupling constraints of the network. In this  problem, we compute the hourly set points of each generator for the whole time horizon.  Additionally, we consider a strongly convex quadratic local cost.

We suppose that the communication links among the agents may fail with certain probability, denoted by $\gamma > 0$. This implies that the activation probability of each communication link is equal, i.e., $\beta_{\{i,j\}}=1-\gamma$, for each $i,j\in \mc N$, where $i\neq j$, and we perform 10 Monte-Carlo simulations for different values of $\gamma$. \color{black}Moreover, we also compare Algorithm \ref{alg:prop} with the unaccelerated version, where $\theta(k)=1$ and $\gamma=0$, for all $k\geq1$. \color{black}Figure \ref{fig:sim_res} shows the convergence of the coupling constraint $\nabla q(\lambda(k))$ toward 0 and the dual value $q(\lambda(k))$ toward the optimal value $q^{\star}$. 
Additionally, Figure \ref{fig:box} shows the number of iterations required to meet the stopping criteria, which is  the error of the equality constraint, i.e., $\|G_i^i u_i(k)+\sum_{j\in \mc N_i}G_i^ju_j(k)-g_i\| < \epsilon$, for a small $\epsilon \geq 0$. 
 As expected, \color{black} Algorithm \ref{alg:prop} significantly outperforms the unaccelerated version, and \color{black} the smaller $\gamma$, the faster the convergence. 


\begin{figure}
	\centering
	\includegraphics[scale=0.42]{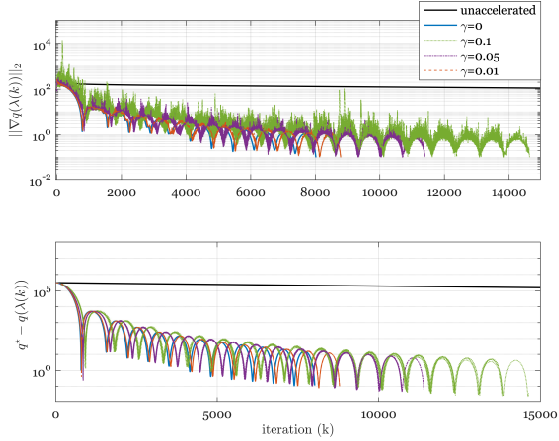}
	\caption{Convergence of $\nabla q(\lambda(k))$ (top) and $q(\lambda(k))-q^{\star}$ (bottom).
	}
	\label{fig:sim_res}
\end{figure}

\begin{figure}
	\centering
	\includegraphics[scale=0.57]{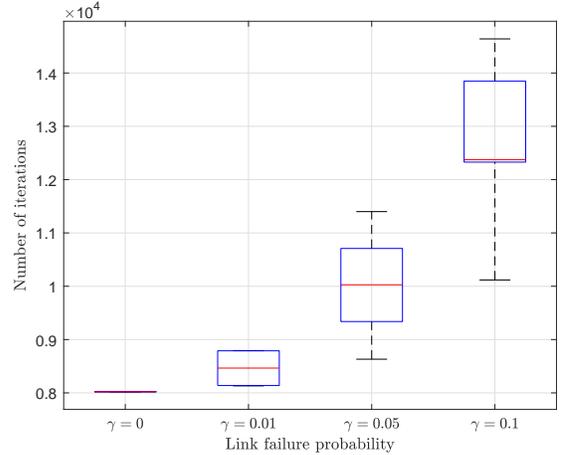}
	\caption{The number of iterations performed for different values of $\gamma$. The blue boxes indicate the 25\textsuperscript{th}-75\textsuperscript{th} percentiles, the red lines indicate the median, and the + symbols indicate the outliers.
	}
	\label{fig:box}
\end{figure}
\section{Conclusion}
\label{sec:concl}
In this paper, we propose a distributed algorithm for multi-agent optimization problem over stochastic networks. The algorithm is based on Nesterov's accelerated gradient method and we analytically show that the convergence rate of the expected dual value is $\mc O(1/k^2)$. We also show the performance of the algorithm in an intra-day optimal power flow simulation. As ongoing work, we are performing an analysis on the convergence of the primal variables. Moreover, we investigate methods to relax the assumptions considered to generalize the approach.

\appendix

\subsection{Proof of Lemma \ref{le:main_ineq}}
\label{sec:app1}

To show Lemma \ref{le:main_ineq}, we can follow the approach used on the proof of \cite[Lemma 2.3]{beck2009}. Therefore, first we need the following intermediate result. 

\begin{lemma}
	\label{le:lb_dif_df}
Let $\psi(\mu,\xi)$ be a quadratic approximation model of $q(\mu)$, i.e.,
\begin{equation}
\label{eq:psi}
	\psi(\mu,\xi) = q(\xi) + \langle \mu-\xi,\nabla q(\xi) \rangle - \sum_{i \in \mc N} \frac{1}{2\eta_i}\|\mu_i-\xi_i\|^2,
\end{equation}
and $\lambda(\xi)$ be defined by
$	\lambda(\xi) = \arg\max_{\mu} \psi(\mu,\xi).$ 
	Furthermore, let Assumptions \ref{as:cost_f}-\ref{as:feas_set} hold and $\eta_i \in (0, 1/L_i]$, where $L_i$ is defined by \eqref{eq:Li}. Then, for any $\mu
 \in \bb R^{\sum_{i \in \mc N}m_i}$, 
	\begin{equation}
	\begin{aligned}
		q(\lambda(\xi))- q(\mu) \geq& \sum_{i \in \mc N}\frac{1}{\eta_i}\langle \xi_i-\mu_i,  \lambda_i(\xi) -\xi_i \rangle\\
		&+ \sum_{i \in \mc N}\frac{1}{2\eta_i}\|\lambda_i(\xi)- \xi_i \|^2.
		\label{eq:lb_dif_df}
	\end{aligned}
	\end{equation}
\end{lemma}
\begin{proof}
	Since $\eta_i \in (0, 1/L_i]$, it follows from Lemma \ref{le:desc} that $q(\lambda(\xi)) \geq \psi(\lambda(\xi),\xi)$. Thus, $$q(\lambda(\xi)) - q(\mu) \geq \psi(\lambda(\xi),\xi)-q(\mu).$$ Since $q(\cdot)$ is concave, we also have that 
	$$q(\mu) \leq q(\xi) + \langle \mu - \xi, \nabla q(\xi)\rangle.$$
	The desired inequality \eqref{eq:lb_dif_df} is obtained by combining the two preceding relations with the definition of $\psi(\lambda(\xi),\lambda)$ in \eqref{eq:psi} and $\lambda(\xi)$.
\end{proof}
\begin{remark}
	The update $\lambda(k)$ in \eqref{eq:lambda_std} follows $\lambda(k) = \arg \max_{\mu} \psi(\mu,\hat{\lambda}(k))$, which admits a unique solution. \eod 
\end{remark}
Next, \cite[Lemma 4]{goldstein2014} shows that $\omega_i(k+1) = \omega_i(k) + \theta(k+1)\left(\lambda_i(k+1)-\hat{\lambda}_i(k+1) \right)$. Based on this relation, we obtain that
\begin{align*}
	&\|\omega_i(k+1)\|^2 - \|\omega_i(k)\|^2 \\
	&= \|\omega_i(k) + \theta(k+1)(\lambda_i(k+1)-\hat{\lambda}_i(k+1))\|^2 - \|\omega_i(k)\|^2\\
	&=2\theta(k+1)(\theta(k+1)-1)\cdot \\ &\quad \cdot \langle{\lambda}_i(k+1)-{\hat{\lambda}}_i(k+1),{\hat{\lambda}}_i(k+1)-{\lambda}_i(k)\rangle+\\
	&\quad + (\theta(k+1)^2-\theta(k+1))\|{\lambda}_i(k+1)-{\hat{\lambda}}_i(k+1)\|^2+\\
	&\quad + \theta(k+1)\|{\lambda}_i(k+1)-{\hat{\lambda}}_i(k+1)\|^2+\\ 
	&\quad +2\theta(k+1)\langle{\lambda}_i(k+1)-{\hat{\lambda}}_i(k+1),{\hat{\lambda}}_i(k+1)-{\lambda}_i^{\star}\rangle,
\end{align*}
where the second equality is obtained by performing some algebraic manipulations using  \eqref{eq:omega} and \eqref{eq:lambda_h_std}. Multiplying by $\frac{1}{2\eta_i}$ and summing over $i\in\mc N$ the above equality, we obtain that
\begin{align*}
	&\sum_{i \in \mc N} \frac{1}{2\eta_i} \left(\|\omega_i(k+1)\|^2 - \|\omega_i(k)\|^2 \right)\\
	& = (\theta(k+1)^2-\theta(k+1))\cdot\\
	&\quad \sum_{i \in \mc N} \left(\frac{1}{\eta_i} \langle{\lambda}_i(k+1)-{\hat{\lambda}}_i(k+1),{\hat{\lambda}}_i(k+1)-{\lambda}_i(k)\rangle\right.\\
	& \qquad\left. + \frac{1}{2\eta_i}\|{\lambda}_i(k+1)-{\hat{\lambda}}_i(k+1)\|^2 \right)\\
	& \quad + \theta(k+1)\sum_{i \in \mc N} \left(\frac{1}{2\eta_i}\|{\lambda}_i(k+1)-{\hat{\lambda}}_i(k+1)\|^2 \right.\\
	& \qquad \left. + \frac{1}{\eta_i}\langle{\lambda}_i(k+1)-{\hat{\lambda}}_i(k+1),{\hat{\lambda}}_i(k+1)-{\lambda}_i^{\star}\rangle \right).
\end{align*}
By applying the inequality \eqref{eq:lb_dif_df} twice to substitute each term inside the two summations, we obtain the desired inequality, as follows:
\begin{align*}
	&\sum_{i \in \mc N} \frac{1}{2\eta_i} \left(\|\omega_i(k+1)\|^2 - \|\omega_i(k)\|^2 \right)\\
	&\leq (\theta(k+1)^2-\theta(k+1))(q(\lambda(k+1)) - q(\lambda(k))\\
	&\quad + \theta(k+1)(q(\lambda(k+1))-q(\lambda^{\star}))\\
	& = \theta(k+1)^2q(\lambda(k+1)) - (\theta(k+1)^2-\theta(k+1))q(\lambda(k))\\
	& \quad - \theta(k+1)q(\lambda^{\star})\\
	& = \theta(k+1)^2q(\lambda(k+1)) - \theta(k)^2q(\lambda(k))\\
	& \quad + (\theta(k)^2-\theta(k+1)^2)q(\lambda^{\star})\\
	& = \theta(k)^2(q(\lambda^{\star})-q(\lambda(k)))\\
	&\quad - \theta(k+1)^2(q(\lambda^{\star})-q(\lambda(k+1))),
\end{align*}
where the second equality is obtained based on step 4 of Algorithm \ref{alg:std}, where $\theta(k+1)^2-\theta(k+1)-\theta(k)^2 = 0$. \eod

\bibliographystyle{ieeetran}
\bibliography{ref}

\end{document}